\documentclass[5p, twocolumn, authoryear]{elsarticle} 

\usepackage[mathscr]{euscript}
\usepackage{amsthm}
\usepackage{amsmath}      

\usepackage{tabularx}

\usepackage{graphics}     
\usepackage{epstopdf}     

\newtheorem{mytheorem}{Theorem}
\newtheorem{mylemma}{Lemma}
\newtheorem{myclaim}{Claim}

\begin{document}


\begin{frontmatter}

\date{}

\title{Bounds on performance measures for the $(r, q)$ lost sales system with Poisson demand and constant lead time}

\author[btj]{Bryan Johnston}
\ead{btj@umich.edu}
\address[btj]{837 Princeton Rd., Berkley, MI 48072}

\begin{abstract}
We demonstrate that the fraction of sales lost for the $(r,q)$ system under consideration can be conveniently bounded in a manner suitable for quick, back-of-the-envelope estimates.
We assume that customer demand arises from a Poisson process with one unit demanded at a time, that all demand occurring during a stockout is lost, and that lead time is constant.
In addition, we allow the situation where multiple replenishment orders may be simultaneously outstanding.
We show that the difference between the upper and lower bounds on the the fraction of sales lost appears likely to be no more than $6.5\%$ ($0.065$), and is typically significantly less than that.
The bounds appear to be relatively tight when the service level is either very high or very low.
In addition, we relate the common performance measures for the system, and so enable bounding the total operating costs.
\end{abstract}


\end{frontmatter}

\tableofcontents


\section{Introduction} \label{SECTION_INTRO}

Lost sales are common in the fast-paced retail environment, and are also used as a mechanism to model emergency orders placed to fulfil customer demand in the event of a stockout (\citet{BIJVANK_2011, DONSELAAR_2013}).
Computerized warehouse systems with automatic bin level tracking now make continuous review of inventory on hand levels a common practice. In addition, increasingly predictable replenishment processes make constant, or nearly constant, lead time a reasonable assumption in many inventory replenishment environments.
We refer the reader to \citet{BIJVANK_2011} for an overview of the lost sales inventory literature.

On the other hand, not much is known about the form of the optimal policy for continuous review systems with lost sales (\citet{BIJVANK_2011}). The ${(r, q)}$ and ${ (s, S) }$ policies are straightforward, easy for inventory controllers to manage, are are thus commonly employed.
Even though such lost sales inventory systems are a classical topic (\citet{HW_1963}), analytic results related to $(r, q)$ lost sales systems are not common in the literature.
Indeed, the $(r, q)$ lost sales system with unit sized and constant lead time is generally acknowledged as particulary difficult to tackle.
Much of this difficulty seems to arise from the fact that, unlike for backorder systems, the lead times remaining on outstanding replenishment orders need to be explicitly taken into account in order to produce the system of integro-differential equations that models the system. The resulting equations have not so far been amenable to analytic solution.

In the literature, three types of simplifications are commonly employed in order to produce analytic results for lost sales systems.
The most common approach is to assume that $r < q$, so that at any particular moment there is at most one outstanding order.
The second approach is to assume that one or both of demand size or lead time is stochastic, for example as in \citet{MOHEBBI_2002}, \citet{JOHANSEN_2004}, and \citet{JOHANSEN_1993}.
The third approach is to assume that the control parameters for the system are set so that the fraction of sales lost is sufficiently small as to be modelled by a backorder system.

Results concerning service level restrictions are also sparse in the literature (\citet{BIJVANK_2011, AARDAL_1989, DONSELAAR_2013}).

Even when analytic results like those mentioned above are available, the expressions involved are not necessarily simple to evaluate or employ in practice.

Thus the need remains to provide simple, closed-form, easy-to-implement analytic results that provide some guarantee for the fraction of sales lost, and other performance measures for $(r, q)$ continuous review systems. To this end, we will examine the $(r, q)$ continuous review system with lost sales, Poisson unit demand, and constant lead time, and derive simple bounds on the fraction of sales lost.

This paper is organized as follows. In Section \ref{SECTION_RPM}, we relate long run fraction of sales lost, average inventory level, average inventory position, and average number of units in outstanding replenishment orders. We show that any one of them determines all the others. In Sections \ref{SECTION_LBG} and \ref{SECTION_UBG}, we obtain lower and upper bounds for the long run fraction of sales lost. In Section \ref{SECTION_NR}, we present some numerical results in order to get a feel for how tight the lower and upper bounds are.

\subsection{Notation}

Table \ref{TABLE_NOTATION} introduces notation used throughout the paper.
The only items which require explanation are $\rho _k$ and $\rho _k (t)$.
In a standard $(r, q)$ system with constant lead time $\tau$, the delivery epoch for the $k$th order (placed at epoch $t_k$) is $\rho_k = t_k + \tau$.
This does not change as the system evolves.
In Section \ref{LTR_SYSTEMS} we allow the remaining lead time on replenishment orders to decrease in response to customer demands. This occurs only when necessary to prevent a stockout situation.
In that context we think of the delivery epoch as a \emph{projected} delivery epoch $\rho_k (t)$, measured at a particular epoch $t$.

\begin{table}[h]
\centering
\caption{
    Notation
} \label{TABLE_NOTATION}
\begin{tabularx}{\columnwidth}{|rX|}
\hline

  $r$                                                 & Reorder point. \\
  $q$                                                 & Order quantity.  \\
  $\tau$                                              & Lead time (constant).  \\
  $\lambda$                                           & Poisson customer demand rate.  \\
  $x$                                                 & $ \lambda \tau $ (mean lead time demand) \\
  \shortstack{
      $\mathscr{S}$,
      $\mathscr{T}$,
      $\overline{\mathscr{S}}$,
      $\overline{\mathscr{T}}$, etc.
  }                                                   & Inventory systems.  \\
  $l(t)$, $l(\mathscr{S}, t)$                         & On hand inventory, at epoch $t$, for system $\mathscr{S}$.  \\
  $u(t)$, $u(\mathscr{S}, t)$                         & Units in outstanding orders, at epoch $t$, for system $\mathscr{S}$. \\
  $p(t)$, $p(\mathscr{S}, t)$                         & Inventory position, at epoch $t$,  for system $\mathscr{S}$. $p(t) = l(t) + u(t)$. \\
  $ \gamma $                                          & Long run fraction of time with no on hand inventory. \\
  $L$, $L(\mathscr{S})$                               & Long run average on hand inventory. \\
  $L(\mathscr{S}, [t_1, t_2])$                        & Average inventory level on the interval $[t_1, t_2]$. \\
  $U$                                                 & Long run average number of units  in outstanding replenishment orders. \\
  $P$                                                 & Long run average inventory position. $P = L + U$. \\
  $\mathscr{O}_k$                                     & The $k$th replenishment order for a system. \\
    $\rho_k$, $\rho_k (t)$                            & The projected delivery epoch, measured at some epoch $t$, of the $k$th replenishment order for a system, and subject to change in \emph{Lead Time Reducing} systems (Section \ref{LTR_SYSTEMS}). \\
  $\omega _k$                                         & The epoch at which $k$th replenishment order was delivered. \\

\hline
\end{tabularx}
\end{table}


\section{Relationship between performance measures $\gamma$, $P$, $L$, $U$} \label{SECTION_RPM}

We are mainly concerned with the $(r, q)$ continuous review system with lost sales, Poisson unit demand, and constant lead time. Let $\mathscr{S}$ be such an $(r, q)$ system.

To relate time spent in the stockout state to inventory position, excise from sample paths of $\mathscr{S}$ all time spent in the stockout state.
These modified sample paths also see Poisson-distributed customer demand, since Poisson demand is forgetful.
Each customer demand occurrence seen by a modified sample path decreases the inventory position by $1$ (modulo $q$).
Thus the inventory position for the modified system is uniformly distributed in ${ r+1, r+2, ..., r+q }$, and has the average value
\\
$$\overline{P} := \frac{1}{q} \Big\{ (r+1) + ... + (r+q) \Big\}  = r + \frac{1}{2}(q+1)$$
\\
The inventory position corresponding to $l(\mathscr{S}, t) = 0$ is $p(\mathscr{S}, t) = q \left\lfloor \frac{\displaystyle r+q}{\displaystyle q} \right\rfloor $, so that
\begin{align}
  P &= (1 - \gamma) \overline{P} + \gamma q \left\lfloor \frac{r+q}{q} \right\rfloor \nonumber \\
    &= (1 - \gamma) \Big( r + \frac{1}{2}(q+1) \Big) + \gamma q \left\lfloor \frac{r+q}{q} \right\rfloor \label{RELATE_P_TO_GAMMA}
\end{align}

The relationship between average on hand inventory, average inventory position, and average number of units in outstanding replenishment orders is well-known and is given by
\\
\begin{align}
  L &= P - U
  \label{RELATE_L_TO_P_AND_U}
\end{align}
\\
See \citet{ZIPKIN_2000, AXSATER_2007, TIJMS_2003, HW_1963}. \\

The relationship between average number of units in outstanding replenishment orders and fraction of time with no on hand inventory is equally straightforward.
Outstanding replenishment orders can be thought of as a queue, with one order added to the queue for each q units of inventory sold.
The average rate orders are placed is $\displaystyle \frac{(1-\gamma) \lambda}{q}$.
The average time spent in the queue is the lead time $\tau$, and so the average number of outstanding orders in the queue is $\displaystyle \frac{(1-\gamma) \lambda \tau}{q}$ by Little's Law (\citet{AXSATER_2007, TIJMS_2003}).
Since there are $q$ units per order, the average number of units in outstanding replenishment orders is

\begin{align}
  U &= (1-\gamma) \lambda \tau     \nonumber \\
    &= (1-\gamma) x
  \label{RELATE_U_TO_GAMMA}
\end{align}
\\
We can now use relationships (\ref{RELATE_L_TO_P_AND_U}), (\ref{RELATE_U_TO_GAMMA}), then (\ref{RELATE_P_TO_GAMMA}) to express the average on hand inventory level in terms of the fraction of time $\gamma$ with no on hand inventory as

\begin{align}
  L &= P - U                         \nonumber \\
    &= P - (1-\gamma) x              \nonumber \\
    &= \Big\{
          (1 - \gamma) \Big( r + \frac{1}{2}(q+1) \Big) + \gamma q \left\lfloor \frac{r+q}{q} \right\rfloor
        \Big\}  \nonumber \\
        &\qquad - (1-\gamma) x              \nonumber \\
  L &= (1 - \gamma) \Big( r + \frac{1}{2}(q+1) - x \Big) + \gamma q \left\lfloor \frac{r+q}{q} \right\rfloor
  \label{RELATE_L_TO_GAMMA}
\end{align}
\\
In summary, we have

\begin{mytheorem} \label{RELATIONS_THEOREM}
The relationships between $\gamma$, $L$, $P$, and $U$ are given by \\
  a) $ \displaystyle P = (1 - \gamma) \Big( r + \frac{1}{2}(q+1) \Big) + \gamma q \left\lfloor \frac{r+q}{q} \right\rfloor $ \\
  b) $ L = P - U $ \\
  c) $ U = (1-\gamma) x $ \\
  d) $ \displaystyle L = (1 - \gamma) \Big( r + \frac{1}{2}(q+1) - x \Big) + \gamma q \left\lfloor \frac{r+q}{q} \right\rfloor $ \\
Any one of $\gamma$, $L$, $P$, or $U$ determines all of the others.
\end{mytheorem}


\section{A lower bound for $\gamma$} \label{SECTION_LBG}

Let $\text{LOSS}(x, r)$ be the \emph{loss function} given by
\\
$$ \text{LOSS}(x, r) = \displaystyle \sum_{k = r}^{\infty}(k - r) \frac{x^k}{k!} e^{-x}$$
\\
For an $(r, q)$ lost sales policy $\mathscr{S}$ where \emph{there can never be more than one simulataneous outstanding order}, the long run fraction of time with no on hand inventory is given by
\\
$$\gamma = \frac{\text{LOSS}(x, r)}{\text{LOSS}(x, r) + q} $$
\\
(see \citet{ZIPKIN_2000, HW_1963}).
In \citet{ZIPKIN_2000}, Zipkin suggests that this may provide a useful approximation in the situation when there may be multiple outstanding orders (when $\left\lfloor \frac{r+q}{q} \right\rfloor > 1$).
When there can only be one outstanding order, $q = q \left\lfloor \frac{r+q}{q} \right\rfloor $, and it leads toward the question of what one can say about
\\
$$ \mathscr{LB} := \frac{\text{LOSS}(x, r)}{\text{LOSS}(x, r) + q \left\lfloor \frac{r+q}{q} \right\rfloor}$$
\\
when $\left\lfloor \frac{r+q}{q} \right\rfloor > 1$.
We will show that it in general provides a lower bound for the fraction of time spent with no on hand inventory.

\begin{mytheorem} \label{LB_THEOREM}
For any $(r, q)$ lost sales system $\mathscr{S}$ with constant lead time $\tau$, Poisson customer demand rate $\lambda$, and $x = \lambda \tau$, we have $\mathscr{LB} \le \gamma $.
This is an equality when $r < q$.
\end{mytheorem}


\subsection{Proof of $\mathscr{LB}$}
In order to prove Theorem \ref{LB_THEOREM}, we will need the next lemma.

\begin{mylemma} \label{LB_LEMMA}
Let $\overline{\mathscr{S}}$ and $\mathscr{S}$ be two identical copies of the same lost sales system, which both see the same customer demand realization, and which both start at the same state at epoch $t = 0$.
Suppose at some epoch $\hat{t}$, the behavior of system $\overline{\mathscr{S}}$ is modified, \emph{one time only}, so that the delivery epoch $\rho_k$ of the earliest pending replenishment order $\mathscr{O}_k$ is reduced from $\rho_k > \hat{t}$ to $\hat{t}$. The remaining lead times on any other outstanding orders remain unchanged.
Order $\mathscr{O}_k$ will then be immediately delivered for system $\overline{\mathscr{S}}$, but not for system $\mathscr{S}$.

Then, the demand served by $\overline{\mathscr{S}}$ in any interval $[0, t]$ is at least as great as the demand served by $\mathscr{S}$ in $[0, t]$.
\end{mylemma}

\begin{proof}
Assume to the contrary, and let $t$ be the smallest epoch such that ${ s[0, t] > \overline{s}[0, t] }$, where $s$ and $\overline{s}$ represent demand served for $\mathscr{S}$ and $\overline{\mathscr{S}}$.
Then ${ s[0, t] = 1 + \overline{s}[0, t] }$.
Since $t$ is the \emph{smallest} such epoch, it's clear that a customer demand must have occurred at epoch $t$, and that system $\mathscr{S}$ served this customer while $\overline{\mathscr{S}}$ did not.
Then for some small $\delta > 0$, immediately before the customer demand at epoch $t$, the inventory levels must have been $l(\mathscr{S}, t - \delta) > 0$ and $l(\overline{\mathscr{S}}, t - \delta) = 0$.
Let $a[0, t]$ and $\overline{a}[0, t]$ denote the number of replenishment units delivered on the interval $[0, t]$ for systems $\mathscr{S}$ and $\overline{\mathscr{S}}$.
At every epoch in $[0, t)$, system $\overline{\mathscr{S}}$ placed its $k$th replenishment order \emph{on or before} the epoch at which $\mathscr{S}$ placed its corresponding $k$th order.
Because lead times are constant, this implies that $a[0, t] \le \overline{a}[0, t]$.
Therefore, immediately after the customer demand at epoch $t$, we have

\begin{align*}
  0 & \le  l(\mathscr{S}, t + \delta)  \\
    & =    l(\mathscr{S}, 0) + a[0, t] - s[0, t] \\
    & =    l(\mathscr{S}, 0) + a[0, t] - (\overline{s}[0, t] + 1) \\
    & \le  l(\mathscr{S}, 0) + \overline{a}[0, t] - \overline{s}[0, t] - 1 \\
    & =    l(\overline{\mathscr{S}}, 0) + \overline{a}[0, t] - \overline{s}[0, t] - 1 \\
    & =    l(\overline{\mathscr{S}}, t) - 1  \\
    & =    -1
\end{align*}

which is a contradiction.

\end{proof}

\noindent
\emph{Remark:} The above argument also works for any customer interarrival distribution, and for any lead time distribution, provided that systems $\mathscr{S}$ and $\overline{\mathscr{S}}$ each receive the same lead time $\tau_k$ for their $k$th replenishment order.

\begin{proof}[Proof of Theorem \ref{LB_THEOREM}]
Consider epochs at which replenishment orders $\mathscr{O}_k$ are placed, with $k \equiv 0 \ (\textrm{mod}\ \left\lfloor \frac{r+q}{q} \right\rfloor)$.
At every such epoch, immediately before the new replenishment order is placed, repeatedly apply Lemma \ref{LB_LEMMA} to force all currently outstanding replenishement orders to arrive.
Call the resulting system $\overline{\mathscr{S}}$.

If $\overline{\mathscr{S}}$ places replenishment order $\mathscr{O}_{m \left\lfloor \frac{r+q}{q} \right\rfloor}$ at epoch $t$, then $l(\overline{\mathscr{S}}, t) = r$ immediately \emph{after} the order is placed.
The system $\overline{\mathscr{S}}$ now has a single outstanding order which has the pending delivery epoch $\rho_{m \left\lfloor \frac{r+q}{q} \right\rfloor} = t + \tau$.

While serving the next $q \left\lfloor \frac{r+q}{q} \right\rfloor$ units until order $\mathscr{O}_{(m + 1) \left\lfloor \frac{r+q}{q} \right\rfloor}$ is placed, lost sales can only occur in the interval $[t, t + \tau]$ while $\mathscr{O}_{m \left\lfloor \frac{r+q}{q} \right\rfloor}$ is outstanding.
For, at epoch $t + \tau + \delta$, for a small $\delta > 0$, we have
\begin{align*}
  l(\overline{\mathscr{S}}, t + \tau + \delta)
      &= r + \overline{a}(t, t + \tau] - \overline{s}(t, t + \tau] \\
      &= r + q - \overline{s}(t, t + \tau] \\
\end{align*}
whereas the amount left to be served before $\mathscr{O}_{(m + 1) \left\lfloor \frac{r+q}{q} \right\rfloor}$ is placed is
\\
$$q \left\lfloor \frac{r+q}{q} \right\rfloor - \overline{s}(t, t + \tau]$$
\\
But $r + q \ge q \left\lfloor \frac{r+q}{q} \right\rfloor$, thus the inventory on hand is at least as great as the amount to be served before order $\mathscr{O}_{(m + 1) \left\lfloor \frac{r+q}{q} \right\rfloor}$ is placed and so no lost sales can occur in the interim.

Then, we see that for system $\overline{\mathscr{S}}$ the number of lost sales that occur while serving $q \left\lfloor \frac{r+q}{q} \right\rfloor$ units is on average $LOSS(\lambda \tau, r)$, since they all occur in the interval of lenth $\tau$ after a replenishment order is placed at reorder point $r$.
For system $\overline{\mathscr{S}}$ then the long run fraction of sales that are lost is
\\
$$\frac{\text{LOSS}(x, r)}{\text{LOSS}(x, r) + q \left\lfloor \frac{r+q}{q} \right\rfloor}$$
\\
Now, a customer demand is a lost sale precisely when the inventory on hand is zero, so by the \emph{PASTA} property of Poisson arrivals (\citet{ZIPKIN_2000, TIJMS_2003}) this is the fraction of time system $\overline{\mathscr{S}}$ spends with no on hand inventory.

From Lemma \ref{LB_LEMMA}, $\overline{\mathscr{S}}$ will serve as much demand as $\mathscr{S}$ on every interval $[0, t]$, so again by \emph{PASTA} $\overline{\mathscr{S}}$ spends on average no more time in a stockout state than $\mathscr{S}$ does.
Then
\\
$$\frac{\text{LOSS}(x, r)}{\text{LOSS}(x, r) + q \left\lfloor \frac{r+q}{q} \right\rfloor} \le \gamma$$
\\
and the proof is complete.

\end{proof}


\section{An upper bound for $\gamma$} \label{SECTION_UBG}

It is well-known (\citet{TAKACS_1969, COHEN_1976}) that for the $(r, 1)$ lost sales system with constant lead time, the fraction of time with no on hand inventory is given by the \emph{Erlang loss formula}
\begin{equation}
\gamma = \frac{ \displaystyle \frac{x^{r+1}}{(r+1)!} }{ \displaystyle \sum_{k = 0}^{r + 1} \frac{x^{k}}{k!} }
          = \frac{ \displaystyle \frac{x^{r+1}}{(r+1)!} }{ \displaystyle \frac{x^{r+1}}{(r+1)!}  + \displaystyle \sum_{k = 0}^{r} \frac{x^{k}}{k!} }
\label{ERLANG_LOSS}
\end{equation}
In this situation, $q = 1$, and $q \left\lfloor \frac{r+q}{q} \right\rfloor = r + 1$, and so $\gamma$ is given by the formula
$$ \mathscr{UB} := \frac{ \displaystyle \left( \frac{r + 1}{q \left\lfloor \frac{r+q}{q} \right\rfloor} \right) \frac{x^{r+1}}{(r+1)!} }
                 { \displaystyle \left( \frac{r + 1}{q \left\lfloor \frac{r+q}{q} \right\rfloor} \right) \frac{x^{r+1}}{(r+1)!} + \displaystyle \sum_{k = 0}^{r} \frac{x^{k}}{k!} } $$
We will show that in general $\mathscr{UB}$ provides an upper bound for the fraction of sales lost.

\begin{mytheorem} \label{UB_THEOREM}
For any $(r, q)$ lost sales system $\mathscr{S}$ with constant lead time $\tau$, Poisson customer demand rate $\lambda$, and $x = \lambda \tau$, we have $\gamma \le \mathscr{UB}$.
This is an equality when $q = 1$.
\end{mytheorem}

Before proving Theorem \ref{UB_THEOREM} we will mention a class of stochastic inventory systems which serve all customer demand with no lost sales or backorders, and which do not have zero lead time.


\subsection{Lead Time Reducing Systems}
\label{LTR_SYSTEMS}

When deriving the relationship (\ref{RELATE_P_TO_GAMMA}) between inventory position $P$ and lost sales $\gamma$, we modified sample paths of $\mathscr{S}$ by removing from them time spent in the state with no on hand inventory.
Using forgetfulness of Poisson demand, these sample paths can be considered to be sample paths of a modified $(r, q)$ lost sales system $\overline{\mathscr{S}}$ described as follows:

\begin{quote}

  The new system $\overline{\mathscr{S}}$ sees the same customer demand realization, and follows the same (r, q) policy as the original system $\mathscr{S}$.
  $\overline{\mathscr{S}}$ also has the same constant lead time $\tau$ as system $\mathscr{S}$, except in the following situation.

  Whenever $l(t) = 1$ and a customer demand arrives at epoch $t$, system $\overline{\mathscr{S}}$ will serve the demand and immediately transition to state $l = 0$ like the usual $(r, q)$ system does.
  If $l = 0$ is a reorder point, $\overline{\mathscr{S}}$ will, like the usual system, place a new replenishment order for $q$ units.
  There will then be $\left\lfloor \frac{r+q}{q} \right\rfloor $ outstanding replenishment orders
  \\
  $$ ( \mathscr{O}_k, \mathscr{O}_{k+1}, ..., \mathscr{O}_{k + \left\lfloor \frac{r+q}{q} \right\rfloor - 1} )$$
  \\
  with projected delivery epochs
  \\
  $$ ( \rho_k, \rho_{k+1}, ..., \rho_{k + \left\lfloor \frac{r+q}{q} \right\rfloor - 1} ) $$
  \\
  The new behavior is: System $\overline{\mathscr{S}}$ will then immediately (upon transitioning to a state with no on hand inventory) reduce the remaining lead times on all oustanding orders by the same amount $\Delta t = \rho_k - t$, where $\rho_k > t$ is the projected delivery epoch of the outstanding order due to arrive next.
  The projected delivery epochs have been changed, and are now
  \\
  $$ ( t, \rho_{k+1} - \Delta t, ..., \rho_{k + \left\lfloor \frac{r+q}{q} \right\rfloor - 1} - \Delta t )$$
  \\
  $\mathscr{O}_k$ is then immediately delivered, and the inventory level increases to $l(t) = q$.
  The system thus spends zero time in the state with no on hand inventory, and serves every customer demand immediately.

\end{quote}

Systems with this behavior will be referred to as \emph{Lead Time Reducing (LTR)} systems.


\subsection{Proof of $\mathscr{UB}$}

The desired bound on $\gamma$ will be recovered from a bound on the average inventory level of a related \emph{LTR} system.
We will construct a sequence of inventory systems whose average inventory levels can be compared, and which in the limit gives the bound in the theorem.

Consider a pair $\mathscr{S}_1, \mathscr{S}_2$ of two \emph{LTR} systems.
Let $\mathscr{O}_{1,k}$ and $\mathscr{O}_{2,k}$ denote the $k$th replenshment orders for $\mathscr{S}_1$ and $\mathscr{S}_2$.
Let $\rho_{1,k}$ and $\rho_{2,k}$ denote projected arrival epochs for $\mathscr{O}_{1,k}$ and $\mathscr{O}_{2,k}$, which are subject to change as the systems evolve (due to the \emph{lead time reducing} behavior of the systems).
When the epoch at which $\rho_{*,k}$ is measured needs to be made explicit we may write $\rho_{*,k}(t)$.
The actual delivery epochs of $\mathscr{O}_{1,k}$ and $\mathscr{O}_{2,k}$ will be indicated by $\omega _{1,k}$ and $\omega _{2,k}$, and are not determined until the epochs at which delivery occurs.
Subscripts $*$, $x$ and $y$ will be used to indicate one or the other of $\mathscr{S}_1$ or $\mathscr{S}_2$, for example $\mathscr{O}_{*,k}$, $\omega _{y,k}$, $\rho _{x,k}$.

\begin{mylemma} \label{UB_LEMMA}
Let $\mathscr{S}_1, \mathscr{S}_2$ be two identical copies of the same $(r, q)$ \emph{LTR} system, which start in the same (unspecified) initial state, and see the same realization of the customer demand process.
Suppose $\mathscr{S}_1$ and $\mathscr{S}_2$ are in an identical state at epoch $t_0$, and at that epoch the systems have (the same) non-zero number of outstanding orders $\mathscr{O}_{*,0}, \mathscr{O}_{*,1}, \cdots, \mathscr{O}_{*,j}$.
\emph{One time only}, reduce the lead times for all outstanding orders for $\mathscr{S}_2$ by the same amount $d_0$, with $0 < d_0 \le \rho_{*,0} - t_0$, while leaving the lead times for system $\mathscr{S}_1$ unchanged.
This is allowed to occur even if $l(\mathscr{S}_2, t_0) > 1$.
Thereafter, both systems follow their usual \emph{LTR} policy.

Then, the average inventory level for $\mathscr{S}_2$ over any interval $[t_0, t]$ is at least as great as the average inventory level for $\mathscr{S}_1$ over that interval.

In other words
$$
L(\mathscr{S}_1, [t_0, t]) \le L(\mathscr{S}_2, [t_0, t])
$$
\end{mylemma}

\begin{proof}[Proof of Lemma \ref{UB_LEMMA}]

A lead time reduction for system $\mathscr{S}_*$ forces order $\mathscr{O}_{*,k}$ to be delivered if it occurs at an epoch when $\mathscr{O}_{*,k}$ is outstanding and is the outstanding order due to arrive next for system $\mathscr{S}_*$.
More precisely, define the left-hand limit
\begin{equation} \label{DEF_RHO_BAR}
  \overline{\rho}_{*,k} = \lim_{t \to {\omega _{*,k}}^- } \rho_{*,k}(t)
\end{equation}
Then $ \overline{\rho}_{*,k} - \omega _{*,k} $ measures the lead time reduction (if any) applied to all outstanding replenishment orders of $\mathscr{S}_*$, in the case a customer arrives when the system has exactly one unit of inventory on hand, and $\mathscr{O}_{*,k}$ is outstanding and next to be delivered.
In other words, $ \overline{\rho}_{*,k} - \omega _{*,k} > 0 $ if a lead time reduction forces the delivery of $\mathscr{O}_{*,k}$, and $ \overline{\rho}_{*,k} - \omega _{*,k} = 0 $ otherwise.

\begin{myclaim} \label{CLAIM_A}
If $ \overline{\rho}_{x,k} - \omega _{x,k} > 0 $ then $\omega _{y,k} \le \omega _{x,k}$.
Thus, a lead time reduction for system $\mathscr{S}_x$ which causes the delivery of $\mathscr{O}_{x,k}$ never results in a situation where $\mathscr{O}_{y,k}$ is left outstanding.
\end{myclaim}

\begin{proof}[Proof of Claim \ref{CLAIM_A}]
Left $\delta > 0$ be small, so that no customer demands or replenishment deliveries occur in the interval $(\omega _{x,k} - \delta, \omega _{x,k})$.
In this interval $l(\mathscr{S}_x, t) = 1$,  with $\mathscr{O}_{x,k}$ outstanding and next to be delivered for system $\mathscr{S}_x$.
Recall that both systems started in the same state, see the same customer demand realization, and serve every unit of demand as it occurs.
Orders $\mathscr{O}_{x,k}$ and $\mathscr{O}_{y,k}$ were placed simultaneously.
If $\mathscr{O}_{y,k}$ is still outstanding in the interval $(\omega _{x,k} - \delta, \omega _{x,k})$, then it is also true that $l(\mathscr{S}_y, t) = 1$ in the interval, and $\mathscr{O}_{y,k}$ is next to be delivered for $\mathscr{S}_y$ (otherwise $l(\mathscr{S}_y, t)  \le 0$).
In this case the customer demand which causes the lead time reduction for $\mathscr{S}_x$ at epoch $\omega _{x,k}$ will also cause a lead time reduction for $\mathscr{S}_y$, and so $\mathscr{O}_{y,k}$ will be delivered at epoch $\omega _{y,k} = \omega _{x,k}$.
\end{proof}

\begin{myclaim} \label{CLAIM_B}
Let $m_*(k)$ be the minimum index $j$ such that order $\mathscr{O}_{*,j}$ was outstanding immediately after $\mathscr{O}_{*,k}$ was placed.
In case $l = 0$ is a reorder point and $\mathscr{O}_{*,k}$ was placed upon hitting $l = 0$, then calculate $m_*(k)$ after $\mathscr{O}_{*,k}$ was placed, but before the lead time reduction was applied. If $m_x(k) < m_y(k)$, then $ \overline{\rho}_{y,j} - \omega _{y,j} = 0 $ for $j = m_x(k), m_x(k) + 1, \cdots , m_y(k) - 1 $.
\end{myclaim}

\emph{Remark:} It's clear that $m_*(k) = k - n$ if the reorder point hit is $r - nq$, for $n = 0, 1, \cdots \left\lfloor \frac{r+q}{q} \right\rfloor - 1$.

\begin{proof}[Proof of Claim \ref{CLAIM_B}]
Let $t$ be the epoch at which orders $\mathscr{O}_{x,k}$ and $\mathscr{O}_{y,k}$ were simultaneously placed.
Then orders $\mathscr{O}_{y,j}$, for $j = m_x(k), m_x(k) + 1, \cdots , m_y(k) - 1 $, were delivered no later than epoch $t$.
If $ \overline{\rho}_{y,j} - \omega _{y,j} > 0 $, then by Claim \ref{CLAIM_A} $\omega _{x,k} \le \omega _{y,k}$, and so the corresponding orders $\mathscr{O}_{x,j}$ would also have been delivered no later than epoch $t$, and that is a contradiction.
\end{proof}

Now, let $t_k$ be the epoch at which order $\mathscr{O}_{*,k}$ was placed.
At epoch $t_k$, $\mathscr{O}_{*,k}$ has the projected delivery epoch $\rho_{*,k}(t_k) = t_k + \tau$.
The projected delivery epoch $\rho_{*,k}$ may decrease, however, as lead time reductions for orders ${ \mathscr{O}_{*,j} }$, ${ j = m_*(k), m_*(k) + 1, \cdots , k - 1 }$, affect the remaining lead time on order $\mathscr{O}_{*,k}$.
We then see that
\begin{align}
  t_k + \tau - \overline{\rho}_{*,k} &= (\overline{\rho}_{*,m_*(k)} - \omega _{*,m_*(k)})            \nonumber \\
                                     & \qquad\quad + (\overline{\rho}_{*,m_*(k) + 1} - \omega _{*,m_*(k) + 1})   \nonumber \\
                                     & \qquad\quad + \cdots                                                      \nonumber \\
                                     & \qquad\quad + (\overline{\rho}_{*, k - 1} - \omega _{*, k - 1})
  \label{TOTAL_LTR_FOR_ORDER}
\end{align}
This measures the total time deducted from the lead time for order $\mathscr{O}_{*,k}$.
We can then express $\overline{\rho}_{1,k} - \overline{\rho}_{2,k}$ as
\begin{align}
  & \overline{\rho}_{1,k} - \overline{\rho}_{2,k} \nonumber \\
  & \quad = \sum_{j=m_2(k)}^{k-1}(\overline{\rho}_{2,j} - \omega_{2,j})
              - \sum_{j=m_1(k)}^{k-1}(\overline{\rho}_{1,j} - \omega_{1,j}) \nonumber \\
  & \quad = \sum_{j=m(k)}^{k-1}(\overline{\rho}_{2,j} - \omega_{2,j})
              - \sum_{j=m(k)}^{k-1}(\overline{\rho}_{1,j} - \omega_{1,j}) 
  \label{DIFF_BETWEEN_RHO_LIMITS}
\end{align}
where $m(k) = \min(m_1(k), m_2(k))$.
The second equality follows from Claim \ref{CLAIM_B}.

In the statement of Lemma \ref{UB_LEMMA}, $d_0$ was defined as some value in the range \\
$0< d_0 \le \rho_{2,0} - t_0$.
We will define $d_k$, for $k > 0$, as
\begin{align}
  d_k = \begin{cases}
           0                    & k < 0 \\
           d_0                  & k = 0 \\
           (\overline{\rho}_{2, k} - \omega _{2, k}) - (\overline{\rho}_{1, k} - \omega _{1, k}) & k > 0
         \end{cases}
  \label{DEF_OF_d_k}
\end{align}

We have arrived at the following result.

\begin{myclaim} \label{CLAIM_C}
We have
\begin{align}
  \text{A)} && \overline{\rho}_{1,k} - \overline{\rho}_{2,k} &= \sum _{j = m(k)} ^ {k-1} d_j \nonumber \\
  \text{B)} && \omega_{1,k} - \omega_{2,k} &= \sum _{j = m(k)} ^ {k} d_j \nonumber
\end{align}
\end{myclaim}

\begin{proof}[Proof of Claim \ref{CLAIM_C}]
\emph{A)} is immediate from (\ref{DIFF_BETWEEN_RHO_LIMITS}).  \emph{B)} follows from \emph{A)} and (\ref{DEF_OF_d_k}).
\end{proof}

\begin{myclaim} \label{CLAIM_D}
For each $d_k$, at least one of
$$0 \le d_k \le \overline{\rho}_{2,k} - \overline{\rho}_{1,k} \text{ \ \ or \ \ }
0 \ge d_k \ge \overline{\rho}_{2,k} - \overline{\rho}_{1,k}$$
is true.
\end{myclaim}

\begin{proof}[Proof of Claim \ref{CLAIM_D}]
If $\omega_{2,k} = \omega_{1,k}$, then $d_k = \overline{\rho}_{2, k} - \overline{\rho}_{1, k}$ and the conclusion holds. Otherwise, suppose $\omega_{1,k} < \omega_{2,k}$. Then $\overline{\rho}_{1, k} - \omega_{1,k} = 0$ by Claim \ref{CLAIM_A}.
Therefore $d_k = \overline{\rho}_{2, k} - \omega_{2,k} \ge 0$.
Since $\omega_{1,k} < \omega_{2,k}$ we have
$\overline{\rho}_{1, k} < \omega_{2,k}$.
It follows that
$$0 \le d_k = \overline{\rho}_{2, k} - \omega_{2,k} < \overline{\rho}_{2, k} - \overline{\rho}_{1, k}$$
The situation when $\omega_{2,k} < \omega_{1,k}$ is similar and yields the other possibility.
\end{proof}

\begin{myclaim} \label{CLAIM_E}
If $k + 1 - \left\lfloor \frac{r+q}{q} \right\rfloor \le j < m(k)$, then $d_j = 0$. Therefore
\begin{align}
  \text{A)} && \overline{\rho}_{1,k} - \overline{\rho}_{2,k} &= \sum _{j = k + 1 - \left\lfloor \frac{r+q}{q} \right\rfloor}  ^ {k-1} d_j \nonumber \\
  \text{B)} && \omega_{1,k} - \omega_{2,k} &= \sum _{j = k + 1 - \left\lfloor \frac{r+q}{q} \right\rfloor} ^ {k} d_j \nonumber
\end{align}
\end{myclaim}

\begin{proof}[Proof of Claim \ref{CLAIM_E}]
Suppose not. Then for at least one system $\mathscr{S}_*$, we see that $\overline{\rho}_{*, j} - \omega_{*,j}$ is non-zero for some index $j$ in
$$
\textstyle \left\{ k + 1 - \left\lfloor \frac{r+q}{q} \right\rfloor, k + 2 - \left\lfloor \frac{r+q}{q} \right\rfloor, \cdots, m(k) - 1 \right\}
$$
This means the system $\mathscr{S}_*$ experienced a lead time reduction at some epoch $t$, forcing the deliver of order $\mathscr{O}_{*,j}$.
Therefore, a customer demand arrived at epoch $t$ when $l(\mathscr{S}_*, t) = 1$.
The system immediately transitioned through $l(\mathscr{S}_*, t) = 0$, at which point there were $\left\lfloor \frac{r+q}{q} \right\rfloor$ outstanding orders.
This implies that order $\mathscr{O}_{*,k}$ was outstanding immediately before order $\mathscr{O}_{*,j}$ was delivered. This is a contradiction, because $j < m(k) \le m_*(k)$.
\end{proof}

If orders for system $\mathscr{S}_2$ tend to be delivered prior to those for system $\mathscr{S}_1$, we should expect the average inventory level for $\mathscr{S}_2$ to be at least as high as that for $\mathscr{S}_1$. To this end let us now define
\begin{equation}
Q(n) = \sum_{k = 0}^n (\omega_{1,k} - \omega_{2,k}) = \sum_{k = 0}^n \left( \sum _{j = k + 1 - \left\lfloor \frac{r+q}{q} \right\rfloor} ^{k} d_j \right)
\label{DEF_1_OF_Q}
\end{equation}
By re-indexing, we can write (\ref{DEF_1_OF_Q}) as
\begin{equation}
Q(n) = \sum _{m = n + 1 - \left\lfloor \frac{r+q}{q} \right\rfloor} ^{n} S(m)
\label{DEF_2_OF_Q}
\end{equation}
where $S(m) = \sum _{k=0}^{m} d_k$.

\begin{myclaim} \label{CLAIM_F}
$S(m) \ge 0$ for all $m \ge 0$, and therefore $Q(n) \ge 0$ for all $n \ge 0$.
\end{myclaim}

\begin{proof}[Proof of Claim \ref{CLAIM_F}]
Suppose not, and let $N$ be the least positive integer such that $S(N) < 0$.
We have $d_N = S(N) - S(N - 1) < 0$, and it follows from Claim \ref{CLAIM_D}
that $\overline{\rho}_{2,N} - \overline{\rho}_{1,N} \le d_N < 0$.
By Claim \ref{CLAIM_C} A) and Claim \ref{CLAIM_E}, we have
$$
- \sum_{j = N + 1 - \left\lfloor \frac{r+q}{q} \right\rfloor} ^ {N-1} d_j \le d_N < 0
$$
Therefore
\begin{align}
S(N - \textstyle \left\lfloor \frac{r+q}{q} \right\rfloor)
  &=   S(N-1) - \sum_{j = N + 1 - \left\lfloor \frac{r+q}{q} \right\rfloor} ^{N-1} d_j \nonumber \\
  &\le S(N -1) + d_N = S(N) \nonumber \\
  &< 0 \nonumber
\end{align}
This is a contradiction, because $N - \left\lfloor \frac{r+q}{q} \right\rfloor < N$.
\end{proof}

The work involved in proving Lemma \ref{UB_LEMMA} is mostly complete, and the remainder is just details.

Let $a_*(t)$ denote the number of replenishment orders for system $\mathscr{S}_*$ which were delivered in the interval $[t_0, t]$, and let $s_*(t)$ denote the number of units of demand served by that system in the same interval. Then
$$
l(\mathscr{S}_*, t) = l(\mathscr{S}_*, t_0) + q a_*(t) - s_*(t)
$$
The two systems $\mathscr{S}_1$ and $\mathscr{S}_2$ were in the same state at epoch $t_0$, and serve customer demands at the same epochs. The difference between the average inventory levels of $\mathscr{S}_1$ and $\mathscr{S}_2$ on the interval $[t_0, t]$ is then given by
\begin{align}
& \frac{1}{t-t_0} \int_{t_0}^{t} \Big( l(\mathscr{S}_2, x) - l(\mathscr{S}_1, x) \Big) \text{\ } dx \nonumber \\
& \qquad \quad = \frac{q}{t-t_0} \int_{t_0}^{t} \Big( a_2(x) - a_1(x) \Big) \text{\ } dx
\label{AVG_LEVEL_CONDITION}
\end{align}
The quantity (\ref{AVG_LEVEL_CONDITION}) will be non-negative if
$$
R(t) = \int_{t_0}^{t} \Big( a_2(x) - a_1(x) \Big) \text{\ } dx \ge 0
$$
\\
When $x < t$, $a_*(x)$ can be expressed as a step function viz:
$$
a_*(x) = \sum_{\omega_{*,k} < t} \mathscr{H}(x - \omega_{*,k})
$$
where $\mathscr{H}(x)$ is the unit step function with $\mathscr{H}(x) = 0$ when $x < 0$ and $\mathscr{H}(x) = 1$ when $x \ge 0$.
Employing this we obtain
\begin{align}
\int_{t_0}^{t} a_*(x) \text{\ } dx
    &= \int_{t_0}^{t} \sum_{\omega_{*,k} < t} \mathscr{H}(x - \omega_{*,k}) \text{\ } dx \nonumber \\
    &= \sum_{\omega_{*,k} < t} \int_{t_0}^{t} \mathscr{H}(x - \omega_{*,k}) \text{\ } dx \nonumber \\
    &= \sum_{\omega_{*,k} < t} (t - \omega_{*,k} ) \nonumber
\end{align}
\\
Hence
\begin{align}
R(t) = \sum_{\omega_{2,k} < t} (t - \omega_{2,k} ) - \sum_{\omega_{1,k} < t} (t - \omega_{1,k} )
\label{R_FORMULA}
\end{align}
\\
Let $M = \max\{k : \omega_{1,k} < t \text{ and } \omega_{1,k} < t \}$.
Then at least one of $\omega_{1,M+1}$ or $\omega_{2,M+1}$ is at least $t$.
If $\omega_{*,M+1} \ge t$, then $\omega_{*,j} \ge t$ when $j \ge M+1$, because orders do not cross in \emph{LTR} systems.
$R(t)$ must then be one of the following forms:
\\
$$
R(t) =
\begin{cases}

  & \sum \limits _{k=0}^{M} (\omega_{1,k} - \omega_{2,k})                                \\

  &        \qquad\qquad \text{ if } \omega_{1,M+1} \ge t \text{ , } \omega_{2,M+1} \ge t          \\[1.8em] 

  & \sum \limits _{k=0}^{M} (\omega_{1,k} - \omega_{2,k})
      + \sum \limits _{ \substack{ M < k \\ \omega_{2,k} < t } } (t - \omega_{2,k})     \\

  &        \qquad\qquad  \text{ if } \omega_{2,M+1} < t \text{ , } \omega_{1,M+1} \ge t            \\[1.8em] 

  & \sum \limits _{k=0}^{M} (\omega_{1,k} - \omega_{2,k})
      - \sum \limits _{ \substack{ M < k \\ \omega_{1,k} < t} } (t - \omega_{1,k})      \\

  &        \qquad\qquad \text{ if } \omega_{1,M+1} < t \text{ , } \omega_{2,M+1} \ge t            \\
\end{cases}
$$
\\
In case $\omega_{1,M+1} \ge t$ and $\omega_{2,M+1} \ge t$, Claim \ref{CLAIM_F} guarantees that $R(t) \ge 0$.

When $\omega_{2,M+1} < t$, we see that the each of the terms $(t - \omega_{2,k})$
with $M < k$ and $\omega_{2,k} < t$ is non-negative. Using this and Claim \ref{CLAIM_F} we see again that $R(t) \ge 0$ in this case.

Finally, suppose $\omega_{1,M+1} < t$. Then $\omega_{2,k} \ge t$ for each of the terms \\
$(t - \omega_{1,k})$ in the right sum with $M < k$.
Then
$$
- \sum \limits _{ \substack{ M < k \\ \omega_{1,k} < t} } (t - \omega_{1,k})
  \ge - \sum \limits _{ \substack{ M < k \\ \omega_{1,k} < t} } (\omega_{2,k} - \omega_{1,k})
$$
Hence
\begin{align}
R(t) & \ge  Q(M) - \sum \limits _{ \substack{ M < k \\ \omega_{1,k} < t} } (\omega_{2,k} - \omega_{1,k}) \nonumber \\
     & \ge  Q(M) + \sum \limits _{ \substack{ M < k \\ \omega_{1,k} < t} } (\omega_{1,k} - \omega_{2,k}) \nonumber \\
     & =  Q(N)  \nonumber \\
     & \ge  0      \nonumber
\end{align}
for some $N$.
This completes the proof of Lemma \ref{UB_LEMMA}.
\end{proof}


Now that we have Lemma \ref{UB_LEMMA} available, we may continue with the proof of Theorem \ref{UB_THEOREM}.
Let $\mathscr{S}$ be the main $(r, q)$ system under consideration in this paper.
From Theorem \ref{RELATIONS_THEOREM} we know that the long run average inventory level is
$$
L(\mathscr{S}) = (1 - \gamma) \Big( r + \frac{1}{2}(q+1) - x \Big) + \gamma q \left\lfloor \frac{r+q}{q} \right\rfloor
$$
It is straightforward to see that the long run average inventory level $\overline{L}(\mathscr{S})$ of $\mathscr{S}$, conditioned on $l(\mathscr{S}, t) > 0$, is given by
\\
\begin{equation} \label{CONDITIONED_AVG_LEVEL_S}
\overline{L}(\mathscr{S})
  = \frac{L(\mathscr{S})}{1-\gamma}
  = r + \frac{1}{2}(q+1) - x + \frac{\gamma}{1 - \gamma} q \left\lfloor \frac{r+q}{q} \right\rfloor
\end{equation}
\\
Now $\gamma \le \mathscr{UB}$ exactly when $\frac{\gamma}{1-\gamma} \le \frac{\mathscr{UB}}{1-\mathscr{UB}}$.
Recalling the definition
$$
\mathscr{UB}
  = \frac{ \displaystyle \left( \frac{r + 1}{q \left\lfloor \frac{r+q}{q} \right\rfloor} \right) \frac{x^{r+1}}{(r+1)!} }
    { \displaystyle \left( \frac{r + 1}{q \left\lfloor \frac{r+q}{q} \right\rfloor} \right) \frac{x^{r+1}}{(r+1)!} + \displaystyle \sum_{k = 0}^{r} \frac{x^{k}}{k!} }
$$
we see that $\frac{\gamma}{1-\gamma} \le \frac{\mathscr{UB}}{1-\mathscr{UB}}$ will be true when
$$
\frac{\gamma}{1-\gamma}
  \le
 \frac{ \displaystyle \left( \frac{r + 1}{q \left\lfloor \frac{r+q}{q} \right\rfloor} \right) \frac{x^{r+1}}{(r+1)!} }
    { \displaystyle \sum_{k = 0}^{r} \frac{x^{k}}{k!} }
$$
\\
Using (\ref{CONDITIONED_AVG_LEVEL_S}), the following is then a sufficient condition for $\gamma \le \mathscr{UB}$:
\begin{align*}
  \textstyle
  \overline{L}(\mathscr{S})
  & \le
    r + \frac{1}{2}(q+1) - x \\
  & \qquad + \left\{
                 \frac{ \textstyle \left( \frac{r + 1}{q \left\lfloor \frac{r+q}{q} \right\rfloor} \right) \frac{x^{r+1}}{(r+1)!} }
                    { \textstyle \sum_{k = 0}^{r} \frac{x^{k}}{k!} }
              \right\}
              q \left\lfloor \frac{r+q}{q} \right\rfloor
\end{align*}
which can be re-written as
\\
\begin{equation} \label{NECC_SUFF_UB_GAMMA}
\overline{L}(\mathscr{S})
  \le
  (r + 1 - x) + (r+1)
    \left\{
      \frac{ \displaystyle \frac{ x^{r+1}}{(r+1)!} } { { \displaystyle \sum_{k = 0}^{r} \frac{x^{k}}{k!} } }
    \right\}
  + \frac{1}{2}(q - 1)
\end{equation}
\\
We will now construct a system whose average inventory level is given by the right-hand side of (\ref{NECC_SUFF_UB_GAMMA}).

Let $\mathscr{U}$ be a \emph{LTR} system with parameters $(r, q = 1, \lambda, \tau)$.
Let $l(\mathscr{U}, t)$ denote the inventory level of $\mathscr{U}$ at epoch $t$, and let $l(\mathscr{U}, t_0) = r + 1$.
We emphasize that system $\mathscr{U}$ places a replenishment order for each unit of demand served.
Now let the reorder quantity $q$ be arbitrary, and consider a $(r, q, \lambda, \tau)$ system $\mathscr{T}$ defined as follows:

\begin{quote}
  Let $l(\mathscr{T}, t_0) = r + 1$, and let $\mathscr{T}$ see the same customer demand realiztion as $\mathscr{U}$.
  If at an epoch $t$ lead times on all outstanding orders for $\mathscr{U}$ are reduced by $\Delta t$, then lead times on all outstanding orders for $\mathscr{T}$ will be reduced by the same amount $\Delta t$. System $\mathscr{T}$ will follow this policy even when $l(\mathscr{T}, t) > 1$. Immediately after lead times are reduced, all orders whose remaining lead times are non-positive will be immediately delivered.
\end{quote}

\emph{Remark:} System $\mathscr{T}$ is not a \emph{LTR} system as defined previously, but is rather a system which experiences lead time reductions which are \emph{driven by} those for system $\mathscr{U}$. In particular, lead time reductions applied to system $\mathscr{T}$ do not necessarily force the delivery of a replishment order for system $\mathscr{T}$.

Let $\mathscr{O}_{n}(\mathscr{U})$ and $\mathscr{O}_{m}(\mathscr{T})$ denote replenishment orders for systems $\mathscr{U}$ and $\mathscr{T}$. Orders $\mathscr{O}_{1}(\mathscr{U})$ and $\mathscr{O}_{1}(\mathscr{T})$ are placed at the same epoch upon both systems simultaneously hitting the reorder point at inventory level $r$ for the first time.
Let $\rho_{n}(\mathscr{U})$ and $\rho_{m}(\mathscr{T})$ denote projected delivery epochs, subject to change due to lead time reductions.
Let $\omega_{n}(\mathscr{U})$ and $\omega_{m}(\mathscr{T})$ denote actual delivery epochs.

\begin{myclaim} \label{CLAIM_G}
Suppose on the interval $[t_0, t_*]$ that $l(\mathscr{U}, t) \le l(\mathscr{T}, t)$ \\
whenever $\nobreak{t \in [t_0, t_*]}$. Then all of the following are true:

\begin{enumerate}

\item[a)] $\mathscr{T}$, like $\mathscr{U}$, serves all customer demand which occurrs in $[t_0, t_*]$ with no lost sales.

\item[b)] If $\mathscr{O}_{kq + 1}(\mathscr{U})$ is placed in the interval $[t_0, t_*]$, then $\mathscr{O}_{kq + 1}(\mathscr{U})$ and $\mathscr{O}_{k + 1}(\mathscr{T})$ are placed at the same epoch.

\item[c)] Any lead time reduction in $[t_0, t_*]$ which reduces $\rho_{kq + 1}(\mathscr{U})$ by $\Delta t$ also reduces $\rho_{k + 1}(\mathscr{T})$ by $\Delta t$.

\item[d)] If $\omega_{kq + 1}(\mathscr{U}) \in [t_0, t_*]$, then $\omega_{kq + 1}(\mathscr{U}) = \omega_{k + 1}(\mathscr{T})$.
Therefore $\mathscr{O}_{kq + 1}(\mathscr{U})$ and $\mathscr{O}_{k + 1}(\mathscr{T})$ are delivered simultaneously.

\item[e)] $0 \le l(\mathscr{T}, t) - l(\mathscr{U}, t) \le q - 1$ for $t \in [t_0, t_*]$.
Furthermore, let $N(t)$ be the largest integer $n$ such that $\omega_{n}(\mathscr{U}) \le t$.
Then $l(\mathscr{T}, t) - l(\mathscr{U}, t) = d$, with $0 \le d \le q - 1$, if and only if $N(t) = kq + (q - d)$.

\end{enumerate}

\end{myclaim}

\begin{proof}[Proof of Claim \ref{CLAIM_G}]

By assumption $l(\mathscr{T}, t) \ge l(\mathscr{U}, t)$ on $[t_0, t_*]$, and system $\mathscr{U}$ serves all customer demand. This shows part $a)$.

System $\mathscr{U}$ places a replenishment order after every unit of demand served, whereas system $\mathscr{T}$ places a replenishment order after every $q$th unit served. This combined with $a)$ shows part $b)$.

Part $c)$ follows from part $b)$ and the definition of system $\mathscr{T}$. Part $d)$ follows immediately from $c)$.

In order to show part $e)$, we note that both systems started at the same inventory level $r + 1$ at epoch $t_0$:
$$
l(\mathscr{U}, t_0) = l(\mathscr{T}, t_0) = r + 1
$$
This, plus part $a)$ shows that the difference in inventory level at a later epoch in $[t_0, t_*]$ is then due only to the difference in number of replenishment units delivered to the two systems.
By part $d)$, we have $\nobreak{\omega_{kq + 1}(\mathscr{U}) = \omega_{k + 1}(\mathscr{T})}$ when $\omega_{kq + 1}(\mathscr{U}) \in [t_0, t_*]$.
Part $e)$ therefore follows from the inequality chain
\begin{align}
\omega_{k + 1}(\mathscr{T})
    &= \omega_{kq + 1}(\mathscr{U})                                        \nonumber \\
    &\le  \omega_{kq + 2}(\mathscr{U})                                     \nonumber \\
    &\le  \cdots                                                           \nonumber \\
    &\le  \omega_{kq + q}(\mathscr{U}) = \omega_{(k + 1)q}(\mathscr{U})    \nonumber \\
    &\le \omega_{(k + 1)q + 1}(\mathscr{U})                                \nonumber \\
    &= \omega_{k + 2}(\mathscr{T})                                         \nonumber
\end{align}
The difference between inventory levels immediately after epoch $ \omega_{k + 1}(\mathscr{T}) = \omega_{kq + 1}(\mathscr{U}) $ is $q - 1$. Each successive delivery for system $\mathscr{U}$ at epochs $\omega_{kq + s}(\mathscr{U})$, $s = 2, 3, \dots$, reduces the difference by one.

\end{proof}

\begin{myclaim} \label{CLAIM_H}
Let $c_1 \le c_2 \le \cdots $ be the epochs at which there is a change in inventory level of either or both of systems $\mathscr{U}$ or $\mathscr{T}$.
If $l(\mathscr{U}, t) \le l(\mathscr{T}, t)$ for $t \in [t_0, c_k]$, then $l(\mathscr{U}, t) \le l(\mathscr{T}, t)$ for $t \in [t_0, c_{k+1}]$.
Therefore by induction $l(\mathscr{U}, t) \le l(\mathscr{T}, t)$ for all $t \ge t_0$.
\end{myclaim}

\begin{proof}[Proof of Claim \ref{CLAIM_H}]

Note that in this proof $l(\mathscr{U}, c_k)$ and $l(\mathscr{T}, c_k)$ are taken to be the inventory levels immediately after the change in level has occurred. We will divide the situation into cases.

First suppose that at epoch $c_k$ the inventory level for $\mathscr{U}$ is strictly less than that for $\mathscr{T}$, i.e. $l(\mathscr{U}, c_k) < l(\mathscr{T}, c_k)$.

If there is no customer demand at epoch $c_{k+1}$, then the delivery of a replenishment order causes the change(s) in inventory level. In this case $ l(\mathscr{U}, c_{k+1}) \le l(\mathscr{T}, c_{k+1}) $.

If on the other hand a customer demand occurring at epoch at epoch $c_{k+1}$ causes the change in level(s), then $l(\mathscr{U}, c_{k+1}) < l(\mathscr{T}, c_{k+1})$, except in the case where $l(\mathscr{U}, c_k) = 1$ and $l(\mathscr{T}, c_k) = 2$.
In that case $l(\mathscr{U}, c_{k+1}) = 1 \le l(\mathscr{T}, c_{k+1})$, because of the lead time reduction applied to system $\mathscr{U}$.

Suppose next that $l(\mathscr{U}, c_k) = l(\mathscr{T}, c_k)$, and that no customer demand occurs at epoch $c_{k+1}$. Then the delivery of a replenishment order causes the change in level(s) at epoch $c_{k+1}$, and this delivery is not triggered by a lead time reduction.
From Claim \ref{CLAIM_G} part $e)$, we see that there is an integer $n$ such that in the interval $(c_k, c_{k+1})$ replenishment orders $\mathscr{O}_{nq}(\mathscr{U})$ and $\mathscr{O}_{n}(\mathscr{T})$ have already been delivered, but that $\mathscr{O}_{nq + 1}(\mathscr{U})$ and $\mathscr{O}_{n + 1}(\mathscr{T})$ have not been delivered.
These orders must have been placed no later than epoch $c_k$.
By Claim \ref{CLAIM_G} part $c)$, we see that the remaining lead times on $\mathscr{O}_{nq + 1}(\mathscr{U})$ and $\mathscr{O}_{n + 1}(\mathscr{T})$ must be equal, at all epochs $t_* \in (c_k, c_{k+1})$.
These two orders will then be delivered simultaneously at epoch $c_{k+1}$, and therefore $l(\mathscr{U}, c_{k+1}) < l(\mathscr{T}, c_{k+1})$.

Finally, suppose that $l(\mathscr{U}, c_k) = l(\mathscr{T}, c_k)$, and that there is a customer demand at epoch $c_{k+1}$.

If $l(\mathscr{U}, c_k) > 1$, then
$$ l(\mathscr{U}, c_{k+1}) = l(\mathscr{U}, c_k) - 1 = l(\mathscr{T}, c_k) - 1 = l(\mathscr{T}, c_{k+1}) $$
and the conclusion holds.

If $l(\mathscr{U}, c_k) = 1$, then $l(\mathscr{T}, t_*) = 1 = l(\mathscr{T}, t_*)$ at all epochs $t_* \in (c_k, c_{k+1})$.
As in a previous case, from Claim \ref{CLAIM_G} part $e)$, we see that there is an integer $n$ such that in the interval $(c_k, c_{k+1})$ replenishment orders $\mathscr{O}_{nq}(\mathscr{U})$ and $\mathscr{O}_{n}(\mathscr{T})$ have already been delivered, but that $\mathscr{O}_{nq + 1}(\mathscr{U})$ and $\mathscr{O}_{n + 1}(\mathscr{T})$ have not been delivered.
These orders must have been placed no later than epoch $c_k$.
By Claim \ref{CLAIM_G} part $c)$, we see that the remaining lead times on $\mathscr{O}_{nq + 1}(\mathscr{U})$ and $\mathscr{O}_{n + 1}(\mathscr{T})$ must be equal, at all epochs $t_* \in (c_k, c_{k+1})$.
The customer demand at epoch $c_{k+1}$ will then, via lead time reduction, trigger the arrival of order $\mathscr{O}_{nq + 1}(\mathscr{U})$.
But then order $\mathscr{O}_{n + 1}(\mathscr{T})$ will be delivered at epoch $c_{k+1}$, since its lead time was reduced by the same amount. Therefore $1 = l(\mathscr{U}, c_{k+1}) < l(\mathscr{T}, c_{k+1}) = q$.

\end{proof}

Now from Claim \ref{CLAIM_H} the inventory levels satisfy $l(\mathscr{U}, t) \le l(\mathscr{T}, t)$ for all $t \ge t_0$, and so the conclusions of Claim \ref{CLAIM_G} are always true.
In particular Claim \ref{CLAIM_G} part $e)$ guarantees that at every epoch $t$, the difference between the inventory
\\
\begin{equation}
l(\mathscr{T}, t) - l(\mathscr{U}, t) = d = (k+1)q - N(t)
\label{DIFF_INV_LEVELS_U_AND_T}
\end{equation}
\\
where $\mathscr{O}_{N(t)}(\mathscr{U})$ is the most recently delivered replenishment order for system $\mathscr{U}$, and $0 \le d \le q - 1$.
We desire the long run distribution of the expression (\ref{DIFF_INV_LEVELS_U_AND_T}).

Let $\pi(k)$, for $k = 0, 1, \dots, q - 1$, denote the long run probability that $\nobreak{ N(t) \equiv k \pmod{q} }$.
System $\mathscr{U}$ has a regeneration epoch (\citet{TIJMS_2003}) at an epoch $t$ whenever $l(\mathscr{U}, t) = r + 1$.
By considering periods of extremely slow customer demand, it is clear that for any pre-chosen $k$, system $\mathscr{U}$ will \emph{almost surely} regenerate at some epoch $t$ with $\nobreak{ N(t) \equiv k \pmod{q} }$.
It follows that $\pi(m + k) = \pi(m)$, for all $m$.
Choosing $k = 1$ shows that $\nobreak{\pi(0) = \pi(1) = \dots = \pi(q-1)}$ and therefore $\pi(k)$ is uniformly distributed with $\pi(k) = \frac{1}{q}$.

We have shown

\begin{myclaim} \label{CLAIM_I}
As $t \to \infty$, $l(\mathscr{T}, t) - l(\mathscr{U}, t)$ is uniformly distributed in $\nobreak{ \{0, 1, \dots, q-1 \} }$. Therefore the long run difference in average inventory levels is given by
$$
L(\mathscr{T}) - L(\mathscr{U}) = \frac{1}{q}\Big[ 0 + 1 + \cdots + (q-1) \Big] = \frac{1}{2}(q - 1)
$$
\end{myclaim}

As a lead time reducing system, $\mathscr{U}$ can be thought of as having been obtained from a standard $\nobreak{ (r, q = 1, \lambda, \tau) }$ lost sales system $\mathscr{U'}$ by excising from sample paths all time spent in a stockout state.

The fraction of time with no on hand inventory for system $\mathscr{U'}$ is given by the \emph{Erlang loss formula} (\ref{ERLANG_LOSS}).
From (\ref{CONDITIONED_AVG_LEVEL_S}) applied to $\mathscr{U'}$, we obtain the formula for the average inventory level for $\mathscr{U}$.
\\
\begin{align}
L(\mathscr{U})
  &= \overline{L}(\mathscr{U'})                                                \nonumber \\
  &= \frac{L(\mathscr{U'})}{1-\gamma(\mathscr{U'})}                            \nonumber \\
  &= (r + 1 - x) + (r+1)\frac{\gamma(\mathscr{U'})}{1 - \gamma(\mathscr{U'})}  \nonumber \\
L(\mathscr{U})
  &= (r + 1 - x) + (r+1)
      \left\{
        \frac{ \displaystyle \frac{ x^{r+1}}{(r+1)!} } { { \displaystyle \sum_{k = 0}^{r} \frac{x^{k}}{k!} } }
      \right\}
  \label{AVG_LEVEL_U}
\end{align}
\\
Combining (\ref{AVG_LEVEL_U}) with Claim \ref{CLAIM_I} we obtain
\\
\begin{equation}
L(\mathscr{T})
=
  (r + 1 - x) + (r+1)
    \left\{
      \frac{ \displaystyle \frac{ x^{r+1}}{(r+1)!} } { { \displaystyle \sum_{k = 0}^{r} \frac{x^{k}}{k!} } }
    \right\}
  + \frac{1}{2}(q - 1)
\label{AVG_LEVEL_T}
\end{equation}
\\
Thus the average inventory level for $L(\mathscr{T})$ given by (\ref{AVG_LEVEL_T}) is the right-hand side of (\ref{NECC_SUFF_UB_GAMMA}). The goal is to show that the inequality in (\ref{NECC_SUFF_UB_GAMMA}) holds for the inventory system $\mathscr{S}$, and so we must show that $\overline{L}(\mathscr{S}) \le L(\mathscr{T})$.

Let $\mathscr{S}$ be the main $\nobreak{ (r, q, \lambda, \tau) }$ system being studied, and let $\overline{\mathscr{S}}$ be a $\nobreak{ (r, q, \lambda, \tau) }$ \emph{LTR} system.
$\overline{\mathscr{S}}$ can be thought of as having been obtained from a standard ${ (r, q, \lambda, \tau) }$ system by excising from sample paths all time spent in a stockout state.
Therefore
\\
$$
\overline{L}(\mathscr{S}) = L(\overline{\mathscr{S}})
$$

Let $\mathscr{S}_1 = \overline{\mathscr{S}}$.
Let $\mathscr{U}$ be the previously considered \emph{LTR} system with parameters ${ (r, q = 1, \lambda, \tau) }$.
Let $\nobreak{ l(\mathscr{S}_1, t_0) = r + 1 = l(\mathscr{U}, t_0) }$, so that both systems start in the same state at epoch $t_0$. In addition, let both systems see the same customer demand realization.

Now, let $t_1$ be the first epoch after $t_0$ where $l(\mathscr{U}, t_1) = 1$ and a customer demand triggers a lead time reduction for $\mathscr{U}$.
Let $\mathscr{S}_2$ be a second copy of $\mathscr{S}_1$, starting in the same state and seeing the same demand realization.
At epoch $t_1$, let $\mathscr{S}_1$ follow its usual lead time reducing policy (only reducing lead times in order to prevent a stockout situation), but let $\mathscr{S}_2$ follow the policy of system $\mathscr{T}$ defined prior to Claim \ref{CLAIM_G}.
System $\mathscr{S}_2$ will then reduce lead times on all outstanding orders by the same amount as system $\mathscr{U}$.
Note that if in fact at epoch $t_1$ system $\mathscr{S}_1$ reduced lead times, then they were reduced by the same amount as for $\mathscr{S}_2$ (Claim \ref{CLAIM_G} parts $c)$ and $e)$).
After epoch $t_1$, let system $\mathscr{S}_2$ return to following the standard lead time reducing policy used by $\mathscr{S}_1$.

We are now in a situation where we may apply Lemma \ref{UB_LEMMA} to systems $\mathscr{S}_1$ and $\mathscr{S}_2$.
We may conclude that the average inventory levels of the two systems over any finite interval $[t_1, t]$ satisfy
$$
L(\mathscr{S}_1, [t_1, t]) \le L(\mathscr{S}_2, [t_1, t])
$$
Since $\mathscr{S}_1$ and $\mathscr{S}_2$ were in precisely the same state at every epoch in $[t_0, t_1)$, it follows that in fact for all $t \ge t_0$
$$
L(\mathscr{S}_1, [t_0, t]) \le L(\mathscr{S}_2, [t_0, t])
$$

Let $t_2$ be the next epoch at which $l(\mathscr{U}, t_2) = 1$ and a customer demand triggers a lead time reduction for $\mathscr{U}$.
Proceeding as before, construct system $\mathscr{S}_3$ which is identical to $\mathscr{S}_2$ on $[t_0, t_2)$, but which at epoch $t_2$ follows the same policy as system $\mathscr{T}$.
At epoch $t_2$ then, $\mathscr{S}_3$ will reduce lead times by the same amount as systems $\mathscr{T}$ and $\mathscr{U}$, whereas system $\mathscr{S}_2$ will follow its usual lead time reducing policy.
After epoch $t_2$, let system $\mathscr{S}_3$ return to following the standard lead time reducing policy used by $\mathscr{S}_1$.
Arguing as before, we have
$$
L(\mathscr{S}_2, [t_0, t]) \le L(\mathscr{S}_3, [t_0, t])
$$

Continuing in this fashion, we obtain a sequence of systems
$$
\overline{\mathscr{S}} = \mathscr{S}_1, \mathscr{S}_2, \mathscr{S}_2, \dots, \mathscr{S}_k, \dots
$$
\\
where $\mathscr{S}_k$ follows the same policy as $\mathscr{T}$ before epoch $t_k$, and follows the same policy as $\overline{\mathscr{S}}$ on and after epoch $t_k$.
For any epoch $t$, there is some $t_k > t$, and therefore on the inteverval $[t_0, t]$ we have
$$
L(\overline{\mathscr{S}}, [t_0, t]) \le L(\mathscr{S}_k, [t_0, t]) = L(\mathscr{T}, [t_0, t])
$$
Since this holds for all finite intervals, it follows that
$$
\overline{L}(\mathscr{S}) = L(\overline{\mathscr{S}}) \le L(\mathscr{T})
$$
and thus the inequality in (\ref{NECC_SUFF_UB_GAMMA}) holds.
This completes the proof of Theorem \ref{UB_THEOREM}.


\section{Numerical Results} \label{SECTION_NR}

Table \ref{TABLE_NUMERIC_RESULTS} shows numerical values for the bounds $\mathscr{LB}$ and $\mathscr{UB}$ over a small range of inputs.
For each reorder point ${ r = 2, 4, 8, \dots, 1024 }$, we evaluated at each reorder quantity ${ q = 2, 3, \dots, r }$, and for mean lead time demand ${ x = \lambda \tau = K r }$ where $K$ ranges through ${ 0.5, 0.75, 1.0, 1.5, 2.0 }$.
The entries in Table \ref{TABLE_NUMERIC_RESULTS} are all aggregates over the range ${ q = 2, 3, \dots, r }$, for fixed values of $r$ and $x$.
The legend below describes the entries in Table \ref{TABLE_NUMERIC_RESULTS}.

\vspace{3mm}
\begin{center}
{ 
  \small
  \begin{tabular}{|l|lc|}
    \hline
    AVG UB SL\%               & Mean value of     & $1 - \mathscr{LB}(r,q,x)$       \\
    AVG LB SL\%               & Mean value of     & $1 - \mathscr{UB}(r,q,x)$       \\
    AVG DIFF                  & Mean value of     & $\mathscr{UB} - \mathscr{LB}$   \\
    MAX DIFF                  & Max value of      & $\mathscr{UB} - \mathscr{LB}$   \\
    MIN DIFF                  & Min value of      & $\mathscr{UB} - \mathscr{LB}$   \\
    \hline
  \end{tabular}
}
\end{center}
\vspace{3mm}

We see that in Table \ref{TABLE_NUMERIC_RESULTS}, the largest average difference (AVG DIFF) occurs at $r = 4$ and $K = 1$ ($\lambda \tau = r = 4$), and that the largest maximum difference (MAX DIFF) occurs at $r = 8$ and $K = 1$ ($\lambda \tau = r = 8$).
When $K$ is kept fixed, the tightness of the service level bounds appears on the whole to improve as $r$ increases. In addition, when $r$ is kept fixed, the bounds appear to improve as $\lambda \tau$ moves away from $r$.

In Figure \ref{WORST_MAX_DIFFS}, we plot the maximum value of MAX DIFF for each reorder point ${ r = 2, 3, \dots, 100 }$.
The factor $K$ was iterated over the range $[0.5, 1.5]$ by increments of $0.01$.
For each reorder point $r$ and each value of $K$, we calculated the MAX DIFF aggregate over the range ${ q = 2, 3, \dots, r }$.
For each $r$, we then selected the maximum value of MAX DIFF over all $K$ considered. For each $r$, this value is the maximum value of $\mathscr{UB} - \mathscr{LB}$ over all $q$ and $K$ considered.
Again it appears likely that the difference between $\mathscr{UB}$ and $\mathscr{LB}$ may tend on the whole to decrease as the reorder point increases.

\begin{figure}[h]
  \caption{Maximum value of MAX DIFF by reorder point}
  \label{WORST_MAX_DIFFS}
  \resizebox{\columnwidth}{!}{
    \includegraphics{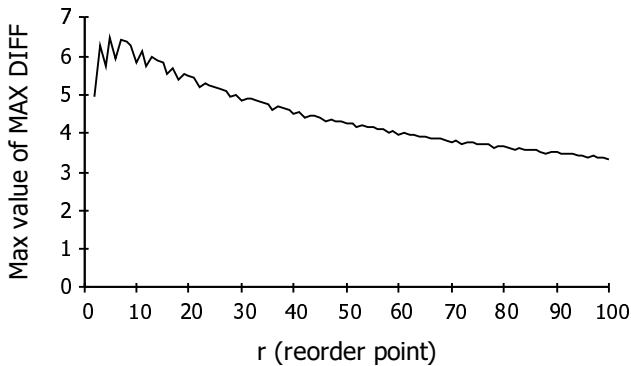}
  }
\end{figure}

\begin{table*}[h]
  \centering
  \caption{
      \newline
      Comparison of service level bounds provided by $\mathscr{LB}$ and $\mathscr{UB}$. \newline
      Entries are aggregates over $q = 2, 3, \dots, r$ for the values of $r$ and $x$ specified by the row and column.
  } \label{TABLE_NUMERIC_RESULTS}
  \resizebox{340pt}{!}{ 
    \begin{tabular}{|r|rrrrrr|}
    \hline
              &             & $K = 0.5$ & $K = 0.75$  & $K = 1$ & $K = 1.5$ & $K = 2$   \\
    \hline
    $r=2$     & AVG UB SL\% & 97.4745   & 93.4371   & 88.0797   & 76.2059   & 65.4676   \\
              & AVG LB SL\% & 95.2381   & 89.5753   & 83.3333   & 71.5789   & 61.9048   \\
              & AVG DIFF    & 2.2364    & 3.8618    & 4.7464    & 4.6270    & 3.5628    \\
              & MAX DIFF    & 2.2364    & 3.8618    & 4.7464    & 4.6270    & 3.5628    \\
              & MIN DIFF    & 2.2364    & 3.8618    & 4.7464    & 4.6270    & 3.5628    \\
    \hline
    $r=4$     & AVG UB SL\% & 98.8653   & 95.3513   & 89.3513   & 74.6444   & 61.8761   \\
              & AVG LB SL\% & 97.1735   & 91.3801   & 84.0794   & 70.0166   & 58.9152   \\
              & AVG DIFF    & 1.6917    & 3.9712    & 5.2718    & 4.6279    & 2.9609    \\
              & MAX DIFF    & 1.8401    & 4.2889    & 5.6346    & 4.8301    & 3.0258    \\
              & MIN DIFF    & 1.3951    & 3.3358    & 4.5463    & 4.2234    & 2.8310    \\
    \hline
    $r=8$     & AVG UB SL\% & 99.7070   & 97.3315   & 91.1333   & 73.4863   & 59.1460   \\
              & AVG LB SL\% & 98.9482   & 94.0050   & 85.9177   & 69.5106   & 57.1057   \\
              & AVG DIFF    & 0.7588    & 3.3265    & 5.2156    & 3.9756    & 2.0404    \\
              & MAX DIFF    & 0.9617    & 4.1430    & 6.2760    & 4.3996    & 2.1109    \\
              & MIN DIFF    & 0.5450    & 2.4457    & 4.0134    & 3.4102    & 1.9060    \\
    \hline
    $r=16$    & AVG UB SL\% & 99.9709   & 98.8840   & 93.2321   & 73.1628   & 58.0321   \\
              & AVG LB SL\% & 99.8342   & 96.7878   & 88.6468   & 70.2620   & 56.8141   \\
              & AVG DIFF    & 0.1366    & 2.0962    & 4.5854    & 2.9008    & 1.2180    \\
              & MAX DIFF    & 0.1656    & 2.5214    & 5.3774    & 3.1609    & 1.2533    \\
              & MIN DIFF    & 0.0932    & 1.4513    & 3.3326    & 2.4177    & 1.1257    \\
    \hline
    $r=32$    & AVG UB SL\% & 99.9994   & 99.6863   & 94.9119   & 72.5549   & 57.0984   \\
              & AVG LB SL\% & 99.9944   & 98.7511   & 91.0831   & 70.7089   & 56.4255   \\
              & AVG DIFF    & 0.0050    & 0.9352    & 3.8289    & 1.8460    & 0.6730    \\
              & MAX DIFF    & 0.0063    & 1.1834    & 4.7048    & 2.0453    & 0.6921    \\
              & MIN DIFF    & 0.0033    & 0.6161    & 2.6455    & 1.5062    & 0.6180    \\
    \hline
    $r=64$    & AVG UB SL\% & 100.0000  & 99.9569   & 96.2880   & 72.2779   & 56.7408   \\
              & AVG LB SL\% & 100.0000  & 99.7382   & 93.2642   & 71.2231   & 56.3845   \\
              & AVG DIFF    & 0.0000    & 0.2187    & 3.0238    & 1.0547    & 0.3563    \\
              & MAX DIFF    & 0.0000    & 0.2776    & 3.7480    & 1.1697    & 0.3661    \\
              & MIN DIFF    & 0.0000    & 0.1413    & 2.0323    & 0.8524    & 0.3262    \\
    \hline
    $r=128$   & AVG UB SL\% & 100.0000  & 99.9984   & 97.3203   & 72.1011   & 56.5268   \\
              & AVG LB SL\% & 100.0000  & 99.9834   & 95.0070   & 71.5312   & 56.3429   \\
              & AVG DIFF    & 0.0000    & 0.0150    & 2.3133    & 0.5698    & 0.1839    \\
              & MAX DIFF    & 0.0000    & 0.0190    & 2.8828    & 0.6320    & 0.1888    \\
              & MIN DIFF    & 0.0000    & 0.0096    & 1.5248    & 0.4579    & 0.1681    \\
    \hline
    $r=256$   & AVG UB SL\% & 100.0000  & 100.0000  & 98.0831   & 72.0381   & 56.4499   \\
              & AVG LB SL\% & 100.0000  & 99.9999   & 96.3572   & 71.7404   & 56.3564   \\
              & AVG DIFF    & 0.0000    & 0.0001    & 1.7259    & 0.2977    & 0.0935    \\
              & MAX DIFF    & 0.0000    & 0.0001    & 2.1551    & 0.3299    & 0.0960    \\
              & MIN DIFF    & 0.0000    & 0.0001    & 1.1246    & 0.2386    & 0.0854    \\
    \hline
    $r=512$   & AVG UB SL\% & 100.0000  & 100.0000  & 98.6328   & 71.9856   & 56.3871   \\
              & AVG LB SL\% & 100.0000  & 100.0000  & 97.3640   & 71.8331   & 56.3400   \\
              & AVG DIFF    & 0.0000    & 0.0000    & 1.2688    & 0.1525    & 0.0471    \\
              & MAX DIFF    & 0.0000    & 0.0000    & 1.5947    & 0.1692    & 0.0484    \\
              & MIN DIFF    & 0.0000    & 0.0000    & 0.8195    & 0.1220    & 0.0430    \\
    \hline
    $r=1024$  & AVG UB SL\% & 100.0000  & 100.0000  & 99.0280   & 71.9620   & 56.3590   \\
              & AVG LB SL\% & 100.0000  & 100.0000  & 98.1062   & 71.8848   & 56.3353   \\
              & AVG DIFF    & 0.0000    & 0.0000    & 0.9218    & 0.0772    & 0.0237    \\
              & MAX DIFF    & 0.0000    & 0.0000    & 1.1614    & 0.0857    & 0.0243    \\
              & MIN DIFF    & 0.0000    & 0.0000    & 0.5919    & 0.0617    & 0.0216    \\
    \hline
    \end{tabular}
    } 
  \label{tab:addlabel}%
\end{table*}


\section{Conclusions} \label{SECTION_CONCL}

In this paper, we have demonstrated that the fraction of sales lost for the $(r,q)$ system under consideration can be conveniently bounded in a manner suitable for quick, back-of-the-envelope estimates.
In Section \ref{SECTION_NR} we saw that the difference between the upper and lower bounds on the the customer fraction of sales lost appears likely to be no more than $6.5\%$ ($0.065$), and in many cases is significantly better than that.
In particular, the estimates we have obtained look to be very useful when the reorder point is relatively high or when the mean lead time demand is significantly different from the reorder point.

An immediate application of the bounds provided in this paper is to refine other common estimates for the fraction of sales lost, such as those in \citet{ZIPKIN_2000} or \citet{HW_1963}.
If, for example, in a particular case the backorder system is used to estimate the fraction of sales lost, and the estimate obtained is more than $\mathscr{UB}$, then one might choose to ignore the backorder estimate and instead use the value $\mathscr{UB}$.

Furthermore, we believe that it should be possible to combine the bounds $\mathscr{LB}$ and $\mathscr{UB}$ in some fashion to develop a new estimate for the fraction of sales lost which improves on other common approximations, while remaining easy-to-compute.

Finally, in Theorem \ref{RELATIONS_THEOREM}, we have related all the performance measures for the system studied.
It is then straightforward to use the bounds $\mathscr{LB}$ and $\mathscr{UB}$ to provide bounds on the total operating costs of the $(r, q)$ continuous review system with lost sales, Poisson unit demand, and constant lead time.


\clearpage

\bibliographystyle{elsarticle-num}
\bibliography{references}

\end{document}